\documentclass[12pt]{amsart}
\usepackage{fullpage,url,mathrsfs,stmaryrd,amssymb,hyperref}

\newtheorem{theorem}{Theorem}[section]
\newtheorem{lemma}[theorem]{Lemma}

\newtheorem{cor}[theorem]{Corollary}

\theoremstyle{definition}
\newtheorem{hypothesis}[theorem]{Hypothesis}

\newtheorem{remark}[theorem]{Remark}
\newtheorem{defn}[theorem]{Definition}
\newtheorem{example}[theorem]{Example}

\numberwithin{equation}{theorem}

\newcommand{\AAA}{\mathbb{A}}

\newcommand{\FF}{\mathbb{F}}
\newcommand{\GG}{\mathbb{G}}

\newcommand{\QQ}{\mathbb{Q}}

\newcommand{\ZZ}{\mathbb{Z}}
\newcommand{\calE}{\mathcal{E}}
\newcommand{\calF}{\mathcal{F}}
\newcommand{\calG}{\mathcal{G}}
\newcommand{\calH}{\mathcal{H}}
\newcommand{\calI}{\mathcal{I}}

\newcommand{\calO}{\mathcal{O}}

\newcommand{\frako}{\mathfrak{o}}

\newcommand{\be}{\mathbf{e}}
\newcommand{\bv}{\mathbf{v}}

\newcommand{\dual}{\vee}

\newcommand{\stacktag}[1]{\cite[\href{https://stacks.math.columbia.edu/tag/#1}{Tag #1}]{stacks-project}}

\DeclareMathOperator{\FIsoc}{\mathbf{F-Isoc}}
\DeclareMathOperator{\FaIsoc}{\mathbf{F}^a\mathbf{-Isoc}}
\DeclareMathOperator{\Aut}{Aut}

\DeclareMathOperator{\coker}{coker}

\DeclareMathOperator{\cts}{cts}
\DeclareMathOperator{\Cts}{Cts}

\DeclareMathOperator{\End}{End}
\DeclareMathOperator{\Ext}{Ext}
\DeclareMathOperator{\FEt}{\mathbf{FEt}}
\DeclareMathOperator{\Frac}{Frac}

\DeclareMathOperator{\Gal}{Gal}
\DeclareMathOperator{\GL}{GL}
\DeclareMathOperator{\Hom}{Hom}
\DeclareMathOperator{\Isoc}{\mathbf{Isoc}}

\DeclareMathOperator{\perf}{perf}

\DeclareMathOperator{\PhiIsoc}{\mathbf{\Phi}-\mathbf{Isoc}}
\DeclareMathOperator{\PhiaIsoc}{\mathbf{\Phi}^\mathit{a}-\mathbf{Isoc}}
\DeclareMathOperator{\PhibIsoc}{\mathbf{\Phi}^\mathit{b}-\mathbf{Isoc}}

\DeclareMathOperator{\rank}{rank}
\DeclareMathOperator{\rig}{rig}

\DeclareMathOperator{\Spec}{Spec}
\DeclareMathOperator{\Spf}{Spf}

\begin{document}

\title{Drinfeld's lemma for $F$-isocrystals, I}
\author{Kiran S. Kedlaya}
\address{Department of Mathematics, University of California San Diego, La Jolla, CA 92093, USA}
\email{kedlaya@ucsd.edu}
\date{30 Jan 2024}
\thanks{The author was supported by NSF grants DMS-1802161 and DMS-2053473, the UCSD Warschawski Professorship, and a Simons Foundation Fellowship for the 2023--24 academic year. 
The author also benefited from the hospitality of MSRI (NSF grant DMS-1928930) during January 2023, CIRM (Luminy) during February 2023, HRIM (Bonn) during June--August 2023, and IAS (Princeton) during September--December 2023. Thanks to Daxin Xu for helpful feedback.}

\begin{abstract}
We prove that in either the convergent or overconvergent setting, an absolutely irreducible $F$-isocrystal on the absolute product of two or more smooth schemes over perfect fields of characteristic $p$, further equipped with actions of the partial Frobenius maps, is an external product of $F$-isocrystals over the multiplicands. The corresponding statement for lisse $\overline{\QQ}_\ell$-sheaves, for $\ell \neq p$ a prime, is a consequence of Drinfeld's lemma on the fundamental groups of absolute products of schemes in characteristic $p$.
The latter plays a key role in V. Lafforgue's approach to the Langlands correspondence for reductive groups with $\ell$-adic coefficients; the $p$-adic analogue will be considered in subsequent work with Daxin Xu.
\end{abstract}

\maketitle

The term ``Drinfeld's lemma'' is used to refer to various statements in a line of inquiry 
arising from Drinfeld's geometric approach to the Langlands correspondence for the group $\GL(2)$ over a global function field of characteristic $p>0$ \cite{drinfeld}. This line of inquiry concerns the geometry of schemes in characteristic $p$, and specifically the effect of taking an absolute product and then forming a (categorical) quotient by the action of some ``partial Frobenius'' maps. For comprehensive introductions to this topic, see \cite{lau} or \cite[Lecture~4]{kedlaya-aws}; the latter includes some related statements for perfectoid spaces (with some details deferred to \cite{ckz} and \cite{kedlaya-simpleconn}).

To give a concrete example of a statement in this circle of ideas, let $X_1, X_2$ be two connected $\FF_p$-schemes and set $X := X_1 \times_{\FF_p} X_2$. The scheme $X$ carries two endomorphisms $\varphi_1, \varphi_2$ obtained by base extension of the absolute Frobenius maps on $X_1, X_2$, respectively. We then have a canonical isomorphism
\begin{equation} \label{eq:drinfeld isom}
\pi_1(X/\varphi_2, \overline{x}) \to \pi_1(X_1, \overline{x}) \times \pi_1(X_2, \overline{x})
\end{equation}
of (profinite) \'etale fundamental groups, where $X/\varphi_2$ denotes the categorical quotient (so that finite \'etale covers of $X/\varphi_2$ correspond to finite \'etale covers of $X$ equipped with an isomorphism with their $\varphi_2$-pullback) and $\overline{x} \to X/\varphi_2$ is a geometric basepoint. (For a proof at this level of generality, see \cite{kedlaya-aws}; see also \cite{lau} for more background.) A key special case is when $X_2$ is itself a geometric point, in which case the conclusion is that $\pi_1(X_1, \overline{x}) \cong \pi_1(X/\varphi_2, \overline{x})$.

The statement given above immediately implies some corollaries for categories of lisse \'etale $\overline{\QQ}_{\ell}$-sheaves, for any prime $\ell \neq p$, using the identification of such sheaves with representations of the \'etale fundamental group(oid); notably,
every absolutely irreducible object on $X/\varphi_2$ is an external product of absolutely irreducible objects on $X_1$ and $X_2$.
With a bit more effort, one can upgrade these to statements about constructible sheaves, as in
\cite[\S 8]{lafforgue-v}; these can then be used to construct \emph{excursion operators} which
play a key role in V. Lafforgue's extension of Drinfeld's construction to reductive groups \cite{lafforgue-v}. Analogous constructions involving perfectoid spaces lead to similar results for the local Langlands correspondence in mixed characteristic \cite{fargues-scholze}.

The purpose of the present paper is to introduce some analogues of the aforementioned statements with lisse  $\overline{\QQ}_{\ell}$-sheaves replaced by corresponding coefficient objects in $p$-adic cohomology. One subtlety is that there are two different, but closely related, categories that can stand in for the category of lisse sheaves: the category of \emph{convergent $F$-isocrystals}, which is easier to construct but does not fit into a six-functors formalism, and the smaller but subtler category of \emph{overconvergent $F$-isocrystals}. The rich interplay between these two categories is a defining feature both of the general theory and of our work here in particular: our overall strategy is to first prove results in the convergent category, where objects are more directly related to representations of fundamental groups, and then transfer knowledge to the overconvergent category.

Our main results are formulated for the categories of \emph{convergent $\Phi$-isocrystals} and \emph{overconvergent $\Phi$-isocrystals} on a product of smooth schemes over perfect fields of characteristic $p$; here $\Phi$ refers to a collection of commuting semilinear actions of the Frobenius maps on the individual factors. As an aid to navigation, we highlight some key points here.
\begin{itemize}
\item
Restriction of convergent $\Phi$-isocrystals to an open dense subscheme in each factor is fully faithful (Lemma~\ref{L:fully faithful}), as is restriction from overconvergent to convergent $\Phi$-isocrystals on a single scheme
(Lemma~\ref{L:fully faithful2}).
\item
One can view a $\Phi$-isocrystal as an $F$-isocrystal by retaining only the product of the Frobenius maps (which we call the \emph{diagonal Frobenius}). One can then refine various standard properties of the diagonal Frobenius: for example, the stratification by diagonal Newton polygons has a product structure
(Theorem~\ref{T:total Newton stratification}),
and a convergent $\Phi$-isocrystal with constant diagonal Newton polygon admits a slope filtration
(Theorem~\ref{T:total slope filtration}).
\item
When one factor is a geometric point, one has an analogue of the Dieudonn\'e--Manin classification theorem to separate slopes in that factor: see Theorem~\ref{T:relative DM} in the convergent case and Theorem~\ref{T:relative DM overconvergent} in the overconvergent case. 
Notably, when $X_2$ is a geometric point,
the pullback functor from $F$-isocrystals over $X_1$ to $X/\varphi_2$ is \emph{not} an equivalence of categories: the essential image of the pullback functor consists of those objects for which $\varphi_2$ acts with all slopes equal to 0.
\item
By repeatedly applying the previous result, we obtain a Dieudonn\'e--Manin decomposition for $\Phi$-isocrystals over a product of geometric points (Corollary~\ref{C:relative DM}). This allows us to construct a \emph{joint Newton polygon stratification} and to produce slope filtrations when the joint Newton polygon is constant (Corollary~\ref{C:partial slope filtration}).
\item
Our main result is that irreducible $\Phi$-isocrystal appears as a constituent of an external product of irreducible $\Phi$-isocrystals, in both the convergent case (Theorem~\ref{T:convergent decomposition}) and the overconvergent case
(Theorem~\ref{T:product overconvergent}).
\end{itemize}

We conclude with a note on coming attractions.
Our motivation for considering Drinfeld's lemma for $F$-isocrystals is the desire to extend  Abe's proof of the Langlands correspondence for $\mathrm{GL}(n)$ with $p$-adic coefficients \cite{abe-companion}, based on the work of L. Lafforgue \cite{lafforgue-l},
to handle reductive groups in the style of \cite{lafforgue-v}. However, this requires some additional preparation, including a reformulation in the language of Tannakian fundamental groups more in the style of \eqref{eq:drinfeld isom};
the latter can be found in a joint paper with Daxin Xu \cite{kedlaya-xu}.

\section{Local Drinfeld's lemma for schemes}
\label{sec:local Drinfeld schemes}

We first state the ``local'' version of Drinfeld's lemma for schemes, in which we consider a product with a geometric point. Here we follow \cite{lau} and \cite[Lecture~4]{kedlaya-aws},
except that the latter treats only finite \'etale coverings and not open immersions.

\begin{defn}
Throughout \S\ref{sec:local Drinfeld schemes},
let $X$ be a scheme over $\FF_p$.
Let $k$ be an algebraically closed field of characteristic $p$.
Put $X_k := X \times_{\FF_p} k$ and let $\varphi_k\colon X_k \to X_k$ be the product of the identity map on $X$ with the absolute Frobenius map on $\Spec k$. 

Let $\FEt(X)$ denote the category of finite \'etale schemes over $X$. We use similar notation for categorical quotients; for example, $\FEt(X_k/\varphi_k)$ will denote the category of finite \'etale schemes over $X_k$ equipped with isomorphisms with their $\varphi_k$-pullbacks.
\end{defn}

\begin{lemma} \label{L:components of product}
The base extension functor
$\FEt(X) \to \FEt(X_k/\varphi_k)$ is an equivalence of categories, with quasi-inverse given by taking $\varphi_k$-invariants.
\end{lemma}
\begin{proof}
See \cite[Lemma~4.2.6]{kedlaya-aws}.
\end{proof}

\begin{lemma} \label{L:open immersion Drinfeld1}
Let $U_k \to X_k$ be a quasicompact open immersion whose image is stable under the action of $\varphi_k$ (meaning that $U_k \times_{X_k,\varphi_k} X_k \to X_k$
factors through $U_k \to X_k$ via an isomorphism). Then 
$U_k = U \times_{\FF_p} k$ for some open subset $U$ of $X$.
\end{lemma}
\begin{proof}
Since $U$ is uniquely determined by $U_k$ if it exists, we may assume at once that $X = \Spec R$ is affine.
Since $U_k$ is quasicompact, its complement is the zero locus of some finitely generated ideal $I = (f_1,\dots,f_n)$ in $R_k := R \otimes_{\FF_p} k$. The fact that $U_k$ is stable under the action of $\varphi_k$ means that there exist a positive integer $m$ and some $n \times n$ matrices $A,B$ over $R_k$ such that
\[
\varphi_k(f_j)^{p^m} = \sum_i A_{ij} f_i, \qquad
f_j^{p^m} = \sum_i B_{ij} \varphi_k(f_i).
\]
Since we can describe $A,B,f_i, \varphi_k(f_i)$ using finitely many elements of $R_k$, we can
find a finitely generated $\FF_p$-subalgebra $R_0$ of $R$ for which $U_k$ arises by pullback from
$\Spec R_0$ to $\Spec R$. That is, we may assume that $X$ is of finite type over $\FF_p$.

Since $X$ is affine and of finite type over $\FF_p$, we may choose a projective compactification $Y$ of $X$.
Let $Z_k$ be the complement of $X_k$ in $Y_k$, viewed as a reduced closed subscheme.
Let $\calI$ be the ideal sheaf cutting out $Z_k$ in $Y_k$; then $\calI$ is isomorphic to its $\varphi_k$-inverse image, so \cite[Lemma~4.2.2]{kedlaya-aws} implies that $\calI$ is the pullback of a coherent sheaf on $Y$.
This immediately yields the desired result.
\end{proof}

\begin{remark}
We do not know if Lemma~\ref{L:open immersion Drinfeld1} remains true without the hypothesis that $U_k \to X_k$ is quasicompact.
This is automatic if $X$ is of finite type over $\FF_p$; in this case, one further deduces a comparison for constructible sheaves as in \cite[\S 8]{lafforgue-v}.
\end{remark}

\section{Products over $\FF_p$}

We continue with some observations about products of schemes over $\FF_p$.

\begin{lemma}  \label{L:geometrically integral product}
For $i=1,\dots,n$, let $X_i$ be a normal integral scheme over $\overline{\FF}_p$. Then
$X_1 \times_{\overline{\FF}_p} \cdots \times_{\overline{\FF}_p} X_n$ is normal and integral.
\end{lemma}
\begin{proof}
Apply \stacktag{09P9} and \stacktag{06DF}.
\end{proof}

\begin{cor} \label{C:geometrically integral product}
For $i=1,\dots,n$, let $X_i$ be a normal integral scheme over $\overline{\FF}_p$
and put $X := X_1 \times_{\FF_p} \cdots \times_{\FF_p} X_n$.
\begin{enumerate}
\item[(a)]
The map
\[
\pi_0(X) \to \pi_0(\overline{\FF}_p \times_{\FF_p} \cdots \times_{\FF_p} \overline{\FF}_p)
\]
is a homeomorphism. In particular, using the action of $\widehat{\ZZ}^n \cong \Gal(\overline{\FF}_p/\FF_p)^n$ on $\overline{\FF}_p \times_{\FF_p} \cdots \times_{\FF_p} \overline{\FF}_p$, we may equip $\pi_0(X)$ with the structure of a principal homogeneous space for the cokernel of the diagonal map
$\Delta\colon \widehat{\ZZ} \to \widehat{\ZZ}^n$.
\item[(b)]
Each connected component of $X$ is integral and normal.
\end{enumerate}
\end{cor}
\begin{proof}
This follows at once from Lemma~\ref{L:geometrically integral product}.
\end{proof}

\begin{remark} \label{R:compare geometric components}
In Corollary~\ref{C:geometrically integral product}, the action of $\widehat{\ZZ}^n$ on $\pi_0(X)$ induces isomorphisms
between the various connected components. However, the resulting isomorphism between any two given components is only 
well-defined up to composition with a power of Frobenius.
\end{remark}

In connection with the previous discussion, we mention the following lemma which we will need later.

\begin{lemma} \label{L:isomorphic representations}
Let $\ell$ be a prime.
Let $E$ be a finite extension of $\QQ_\ell$.
Let $\Cts(\widehat{\ZZ}^n, E)$ be the set of continuous maps of $\widehat{\ZZ}^n$ to $E$, viewed as a topological ring carrying the compact-open topology.
Let $G$ be a profinite group. Let $\rho\colon G \to \GL_m(\Cts(\widehat{\ZZ}^n, E))$ be a continuous homomorphism.
For each $x \in \widehat{\ZZ}^n$, let $\rho_x\colon G \to \GL_m(E)$ be the composition of $\rho$ with the evaluation-at-$x$ map on
$\Cts(\widehat{\ZZ}^n, E)$. 
Suppose that there exist isomorphisms of $\rho$ with its pullbacks along the translation map on $\widehat{\ZZ}^n$ by each of the generators.
Then the representations $\rho_x$ are pairwise isomorphic for all $x \in \widehat{\ZZ}$;
moreover, there exists a matrix $h \in \GL_m(\Cts(\widehat{\ZZ}^n, E))$ such that $\rho$ has values in $h^{-1} \GL_m(E) h$.
\end{lemma}
\begin{proof}
The character of $\rho$ corresponds to a function $G \to \Cts(\widehat{\ZZ}^n, E)$ which is continuous and $\ZZ^n$-equivariant, and hence factors through $E$;
by the Brauer--Nesbitt theorem \cite[Corollary~XVII.3.8]{lang}, the representations $\rho_x$ must have pairwise isomorphic semisimplifications.

We now proceed by induction on $m$,
with the base case being when the $\rho_x$ are all irreducible.
In this case, we already know that the $\rho_x$ are pairwise isomorphic.
By Schur's lemma, $F := \End(\rho_0)$ is a division algebra containing $E$.
Define the quotient $Q := (E^m - \{0\})/ F^\times$. For each $x \in \widehat{\ZZ}^n$, the isomorphisms between $\rho_x$ and $\rho_0$
form an $F^\times$-torsor; we then obtain a unique element $f_x \in Q$ by taking the standard basis of $E^m$ (viewed as the target of $\rho_x$) and collecting its images under all isomorphisms with $\rho_0$.
By continuity of the matrix entries, we can find a neighborhood $U$ of $0$ in $\widehat{\ZZ}^n$
such that $E^m$ admits a lattice $T$ stable under $\rho_x$ for all $x \in U$ (compare \cite[Proposition~1.15]{bckl}).
Using this lattice, we may deduce that the $f_x$ define a continuous map $U \to Q$;
after possibly shrinking $U$ further, we may lift to obtain an element $h \in \GL_m(\Cts(U, E))$ such that $\rho|_U$ has values in $h^{-1} \GL_m(E) h$. By translation, we may obtain a similar conclusion with $U$ taken to be a neighborhood of any fixed element of $\widehat{\ZZ}^n$; by compactness, we can cover $\widehat{\ZZ}^n$ with finitely many such neighborhoods. We thus obtain an element $h \in \GL_m(\Cts(U, E))$ such that $\rho|_U$ has values in $h^{-1} \GL_m(R) h$ where $R$ is the ring of locally constant functions $\widehat{\ZZ}^n \to E$; once again using the isomorphisms of $\rho_0$ with $\rho_x$, we may further adjust $h$ 
to ensure that $\rho|_U$ has values in $h^{-1} \GL_m(E) h$.

For the induction step, suppose that the $\rho_x$ are all reducible. 
Let $\iota_1,\dots,\iota_n$ be the specified isomorphisms of $\rho$ with its pullbacks along translations by the standard generators of $\widehat{\ZZ}^n$.
On account of the induction hypothesis,
it suffices to verify the following:
if there is a short exact sequence
\[
0 \to V_1 \to E^m \to V_2 \to 0
\]
which is stable under the action of $\rho_x$ for all $x$ and preserved by $\iota_1,\dots,\iota_n$, 
and such that the induced maps $\rho_i\colon G \to \Cts(\widehat{\ZZ}^n, \GL(V_i))$ factor through
the subgroup of constant maps in $\Cts(\widehat{\ZZ}^n, \GL(V_i))$, then the induced
map $\widehat{\ZZ}^n \to \Ext^1_{\cts}(V_2, V_1)$ must be constant.
To check this, let $T_1, T_2$ be $G$-stable lattices in $V_1, V_2$, so that $\Ext^1_{\cts}(V_2, V_1) = \Ext^1_{\cts}(T_2, T_1) \otimes_{\frako_E} E$.
Note that $\Ext^1_{\cts}(T_2, T_1)$ is torsion-free (because $H^1_{\cts}(G, \frako_E) = \Hom_{\cts}(G, \frako_E)$ is)
and Hausdorff (because $\Hom(T_2, T_1)$ is complete).
Hence the map $\widehat{\ZZ}^n \to \Ext^1_{\cts}(V_2, V_1)$ is a continuous, $\ZZ^n$-equivariant map into a Hausdorff topological $E$-vector space, and hence must be constant.
\end{proof}

\section{Drinfeld's lemma for schemes}
\label{sec:Drinfeld schemes}

We continue with the global version of Drinfeld's lemma for schemes.

\begin{defn}
Throughout \S\ref{sec:Drinfeld schemes},
for $i=1,\dots,n$, let $X_i$ be a connected scheme over $\FF_p$.
Put $X = X_1 \times_{\FF_p} \cdots \times_{\FF_p} X_n$.
For $i=1,\dots,n$, let $\varphi_i\colon X \to X$ be the map that acts by Frobenius on the $i$-th factor and the identity on the other factors; note that these maps commute pairwise.
\end{defn}

\begin{theorem} \label{T:Drinfeld}
Suppose that $X_1,\dots,X_n$ are connected and qcqs, and fix a geometric point $\overline{x} \to X$. 
Then we have a canonical isomorphism of topological groups
\[
\pi_1(X/\Phi, \overline{x}) \cong \pi_1(X_1/\varphi_1, \overline{x}) \times_{\pi_1(\FF_p/\varphi, \overline{x})} \cdots \times_{\pi_1(\FF_p/\varphi, \overline{x})}
\pi_1(X_n/\varphi_n, \overline{x})
\]
where $\Phi$ denotes the group of automorphisms of $X$ generated by $\varphi_1,\dots,\varphi_n$.
\end{theorem}
\begin{proof}
See \cite[Theorem~4.2.12, Remark~4.2.15]{kedlaya-aws}.
Alternatively, see \cite[Lemme~8.11]{lafforgue-v} for a reformulation that also applies when the $X_i$ are not required to be qcqs.
\end{proof}

We formulate a corollary in terms of $\ell$-adic sheaves that will serve as a model for our $p$-adic results (Theorems~\ref{T:convergent decomposition} and~\ref{T:product overconvergent}).
\begin{lemma} \label{L:external product decomposition}
Let $\ell$ be a prime.
Let $G_1,\dots,G_n$ be profinite groups and put $G := G_1 \times \cdots \times G_n$.
Let $\rho\colon G \to \GL(V)$ be a continuous irreducible linear representation  on a finite-dimensional $\overline{\QQ}_\ell$-vector space. Then there exist continuous irreducible linear representations $\rho_i\colon G_i \to \GL(V_i)$ on finite-dimensional $\overline{\QQ}_\ell$-vector spaces for which $\rho \cong \rho_1 \boxtimes \cdots \boxtimes \rho_n$.
\end{lemma}
\begin{proof}
It suffices to treat the case $n=2$, as the general case then follows by induction.
Let $\rho_1\colon G_1 \to \GL(V_1)$ be an irreducible subrepresentation of $\rho|_{G_1}$.
Then $W := (V_1^\dual \otimes V)^{G_1}$ is a finite-dimensional $\overline{\QQ}_\ell$-vector space carrying a continuous action of $G_2$ via the action on $V$ (i.e., the action on $V_1$ is taken to be trivial). Let $\rho_2\colon G_2 \to \GL(V_2)$ be an irreducible subrepresentation of $W$. 
Then $(V_2^\dual \otimes W)^{G_2} \neq 0$ and so $(V_1^\dual \otimes V_2^\dual \otimes V)^G \neq 0$.
That is, there exists a nonzero $G$-equivariant map $V_1 \otimes V_2 \to V$, which must be surjective because $V$ is irreducible.
In fact $V_1 \otimes V_2$ is also irreducible: by Schur's lemma,
\[
((V_1 \otimes V_2)^\dual \otimes (V_1 \otimes V_2))^G =(V_1^\dual \otimes V_1)^{G_1} \otimes_{\overline{\QQ}_\ell} (V_2^\dual \otimes V_2)^{G_2} =
\overline{\QQ}_\ell \otimes_{\overline{\QQ}_\ell} \overline{\QQ}_\ell  = \overline{\QQ}_\ell.
\]
Hence $V_1 \otimes V_2 \to V$ must be an isomorphism, as claimed.
\end{proof}

\begin{cor} \label{C:external product decomposition lisse}
Suppose that $X_1,\dots,X_n$ are qcqs.
Let $\ell \neq p$ be a prime.
Let $\calE$ be an irreducible object in the category of lisse $\overline{\QQ}_\ell$-sheaves on $X$ equipped with 
isomorphisms $F_i\colon \varphi_i^* \calE \cong \calE$ for $i=1,\dots,n$ that commute pairwise in the
sense that $(\varphi_j^* F_i) \circ F_j$ and $(\varphi_i^* F_j) \circ F_i$ coincide as morphisms
from $\varphi_i^* \varphi_j^* \calE \cong \varphi_j^* \varphi_i^* \calE$ to $\calE$.
Then there exist irreducible lisse $\overline{\QQ}_\ell$-sheaves $\calE_i$ on $X_i$ such that $\calE \cong \calE_1 \boxtimes \cdots \boxtimes \calE_n$.
\end{cor}
\begin{proof}
Via Theorem~\ref{T:Drinfeld}, this reduces to Lemma~\ref{L:external product decomposition}.
\end{proof}

\begin{lemma} \label{L:closed subspace Drinfeld1}
Suppose that $n=2$ and that 
$X_i$ is projective over $\FF_p$. Let $Z \to X$ be a closed immersion with $Z$ integral,
whose image is stable under the action of $\varphi_i$ for $i=1,2$
(meaning that $Z \times_{X,\varphi_i} X \to X$ factors through $Z \to X$ via an isomorphism).
Then there exist integral closed subschemes $Z_i$ of $Z$ such that $Z = Z_1 \times_{\FF_p} Z_2$.
\end{lemma}
\begin{proof}
The projection $X \to X_i$ is a closed map, so the image of $Z$ in $X_i$ is closed; by replacing the $X_i$ with suitable closed subschemes, we may assume that $Z$ surjects onto $X_1$ and $X_2$. We may also assume that $X_1$ and $X_2$ are themselves integral; under these conditions, we show that $Z = X$.

Let $\overline{\eta}_1$ be a geometric point of $X_1$ lying over the generic point $\eta_1$.
By Lemma~\ref{L:open immersion Drinfeld1}, $Z \times_X (\overline{\eta}_1 \times_{\FF_p} X_2)$ has the form
$\overline{\eta}_1 \times_{\FF_p} Z_2$ for some reduced closed subscheme $Z_2$ of $X_2$. Since the image of $Z$
in $X_2$ is contained in $Z_2$, in order for $Z \to X_2$ to be surjective
we must have $Z_2 = X_2$; that is, $Z$ contains
$\eta_1 \times_{\FF_p} X_2$. Since the latter is dense in $X$, we deduce that $Z = X$ as desired.
\end{proof}

\begin{theorem} \label{T:stable open}
Let $U \to X$ be a quasicompact open immersion whose image is stable under the action of $\varphi_i$ for $i=1,\dots,n$ (meaning that $U \times_{X,\varphi_i} X \to X$ factors through $U \to X$ via an isomorphism).
Then $U$ is covered by subschemes of the form $U_1 \times_{\FF_p} \cdots \times_{\FF_p} U_n$
where $U_i$ is an open affine subscheme of $X_i$.
\end{theorem}
\begin{proof}
By induction, we may reduce to the case $n=2$. We may also assume that $X_1, X_2$ are affine.
As in the proof of Lemma~\ref{L:open immersion Drinfeld1}, we may further assume that 
$X_1, X_2$ are of finite type over $\FF_p$.
We may now choose projective compactifications $Y_1, Y_2$ of $X_1, X_2$ and view $U$ as an open subscheme
of $Y := Y_1 \times_{\FF_p} Y_2$. By Lemma~\ref{L:closed subspace Drinfeld1}, each irreducible component of
$Y \setminus U$ is of the form $Z_1 \times_{\FF_p} Z_2$ where $Z_i$ is a closed subset of $X_i$; this proves the claim.
\end{proof}

\section{Convergent $\Phi$-isocrystals}

We next introduce the category of convergent $\Phi$-isocrystals on a product of $\FF_p$-schemes, each smooth over some perfect field. The reader is invited to keep in mind the case where all of these base fields are finite (or even $\FF_p$), as this is the most crucial case for applications (see Remark~\ref{R:finite base field});
however, for certain arguments it will be convenient to reduce to the case where these fields are all algebraically closed.

\begin{defn} \label{D:product notation}
Hereafter, for $i=1,\dots,n$, let $k_i$ be a perfect field of characteristic $p$,
let $K_i$ be the fraction field of the Witt ring over $k_i$,
and put $K := K_1 \widehat{\otimes}_{\QQ_p} \cdots \widehat{\otimes}_{\QQ_p} K_n$.
(That is, take the $p$-adic completion of $W(k_1) \otimes_{\ZZ_p} \cdots \otimes_{\ZZ_p} W(k_n)$, then invert $p$ to obtain $K$. Note that $K$ is not a field if $i > 1$.)

Let $X_i$ be a smooth scheme over $k_i$ and put
\[
X := X_1 \times_{\FF_p} \cdots \times_{\FF_p} X_n.
\]
Let $\varphi_i\colon X \to X$ be the map that acts by Frobenius on the $i$-th factor and the identity on the other factors; note that these maps commute pairwise.
\end{defn}

\begin{remark} \label{R:finite base field}
Beware that in general $X$ is not noetherian, which creates some complications in what follows. However, in the crucial use case of the Langlands correspondence, the base fields $k_i$ are finite,  in which case this and many other technical difficulties evaporate. See \cite{kedlaya-xu} for a more detailed treatment of this case.
\end{remark}

\begin{defn} \label{D:phi-isocrystal}
For $a$ a positive integer,
we define the category $\PhiaIsoc(X)$ of \emph{convergent $\Phi^a$-isocrystals} on $X$
using local liftings as in \cite[\S 2]{kedlaya-isocrystals}; 
it is a stack for the Zariski topology and the \'etale topology on each $X_i$.

To make this explicit, assume that each $X_i$ is affine; we can then choose a smooth affine formal scheme $P_i$ over $W(k_i)$
with special fiber $X_i$, together with a Frobenius lift $\sigma_i$ on $P_i$. Let $P$ be the product
$P_1 \times_{\ZZ_p} \cdots \times_{\ZZ_p} P_n$.
We define the module of continuous differentials $\Omega^1_P$
as the external product $\Omega_{P_1/W(k_1)} \boxtimes \cdots \boxtimes \Omega_{P_n/W(k_n)}$; in particular, it is a finite projective $\Gamma(P, \calO)$-module.

An object of $\PhiaIsoc(X)$ may then be interpreted as a finite projective module $\calE$ over $\Gamma(P, \calO)[p^{-1}]$
equipped with an integrable $K$-linear connection
$\nabla\colon \calE \to \calE \otimes_{\Gamma(P,\calO)} \Omega^1_P$
and horizontal isomorphisms $F_i\colon (\sigma_i^a)^* \calE  \cong \calE$ that commute pairwise in the sense that 
for $1 \leq i,j \leq n$, $((\sigma_j^a)^* F_i) \circ F_j$ and $((\sigma_i^a)^* F_j) \circ F_i$ coincide as morphisms from
$(\sigma_i^a)^* (\sigma_j^a)^* \calE \cong (\sigma_j^a)^* (\sigma_i^a)^* \calE$ to $\calE$.

Using similar arguments to the classical case, one can show that the definition of $\PhiaIsoc(X)$ is independent of the choice of $X$; more precisely, one defines a site of formal lifts and shows that the pullback functor from crystals on this site to $\PhiaIsoc(X)$ as we have defined it is an equivalence.

The category $\PhiaIsoc(X)$ is a $\QQ_{p^a}$-linear tensor category, where $\QQ_{p^a}$ denotes the unramified
extension of $\QQ_p$ with residue field of degree $a$ over $\FF_p$.
\end{defn}

\begin{remark} \label{R:reduce to first power}
For $b$ divisible by $a$, the restriction functor from
$\PhiaIsoc(X)$ to $\PhibIsoc(X)$ admits a left and right adjoint given by
\[
\calE \mapsto \bigoplus_{i_1,\dots,i_n=0}^{b/a-1} (\sigma_1^{ai_1})^* \cdots (\sigma_n^{ai_n})^* \calE.
\]
We will use these functors at various times to make reductions in both directions, i.e., to force $a=1$ or to force $a$ to be conveniently large.
\end{remark}

\begin{defn} \label{D:absolutely irreducible}
In connection with Remark~\ref{R:reduce to first power}, we say that an object $\calE \in \PhiaIsoc(X)$ is \emph{absolutely irreducible}
if its restriction to $\PhibIsoc(X)$ remains irreducible for any multiple $b$ of $a$.

Note that if $\calE$ is irreducible but not absolutely irreducible, then $H^0(X, \calE^\dual \otimes \calE)$ is a division algebra over $\QQ_{p^a}$, which is split by $\QQ_{p^b}$ for some multiple $b$ of $a$ by local class field theory.
For this reason, it will be sufficient for our purposes to consider $\Phi^a$-isocrystals for varying $a$, without introducing ramified extension of $\QQ_p$ into the coefficients.
\end{defn}

\begin{lemma} \label{L:Fitting ideals}
Suppose that $X_1,\dots,X_n$ are geometrically connected and affine and let $\calO$ be the unit object of $\PhiaIsoc(X)$.
Then the only nonzero submodule of $\Gamma(P, \calO)$ stable under
the Frobenius and connection structures is $\calO$ itself.
\end{lemma}
\begin{proof}
Using Remark~\ref{R:reduce to first power}, we reduce to the case $a=1$.
Let $M$ be a submodule of $\Gamma(P, \calO)$ stable under
the Frobenius and connection structures of $\calO$.
Choose a field extension $k$ of $k_1$ for which there exists a $k$-valued point
$x \to X_2 \times_{\FF_p} \cdots \times_{\FF_p} X_n$. We can then
pull back $\calE$ along the map $X_1 \times_{k_1} k \to X_1 \times_{\FF_p} \cdots \times_{\FF_p} X_n$ induced by
the projection $X_1 \times_{k_1} k \to X_1$ and the map $X_1 \times_{k_1} k \to \Spec k \cong x \to X_2 \times_{\FF_p} \cdots \times_{\FF_p} X_n$; using the Frobenius action coming from the composition of the actions of the $\varphi_i$,
we obtain a submodule of the unit object in $\FIsoc(X_1 \times_{k_1} k)$.
Since $X_1$ is geometrically connected, so is $X_1 \times_{k_1} k$, so the resulting object is either zero or the unit object (because a coherent sheaf on a rigid analytic space admitting an integrable connection is locally free; see for example \cite[Lemma~3.3.3]{kedlaya-goodformal2}). 

A similar argument applies with $X_1$ replaced by $X_i$ for $i=2,\dots,n$. Using these facts, we see that if $M$ is nonzero, then it is supported at every point of $\Spec \Gamma(P, \calO)[p^{-1}]$, and thus must be equal to $\calO$.
\end{proof}

\begin{cor} 
The category $\PhiaIsoc(X)$ is abelian.
\end{cor}
\begin{proof}
Given a morphism $f\colon \calE_1 \to \calE_2$ in $\PhiaIsoc(X)$, by applying Lemma~\ref{L:Fitting ideals} to the Fitting ideals of the cokernel of $f$, we see that this cokernel is a projective module over $\Gamma(P, \calO)[p^{-1}]$.
This implies in turn that the image of $f$ is projective, as then is the kernel.
\end{proof}

\begin{lemma} \label{L:mod p restrict to point}
Let $R$ and $S$ be reduced $\FF_p$-algebras and choose $g \in S$ not a zero divisor. Then
an element $f \in R \otimes_{\FF_p} S[g^{-1}]$ belongs to $R \otimes_{\FF_p} S$ if and only if
for every morphism $R \to k$ with $k$ an algebraically closed field, the image of $f$ in
$k \otimes_{\FF_p} S[g^{-1}]$ belongs to $k \otimes_{\FF_p} S$.
\end{lemma}
\begin{proof}
Choose a basis for the $\FF_p$-vector space $S[g^{-1}]/S$, then write the image of $f$ in $R \otimes_{\FF_p} (S[g^{-1}]/S)$
in terms of this basis. By considering individual basis vectors separately, we reduce to the observation that $R$ injects into a product of fields.
\end{proof}

\begin{lemma} \label{L:restrict to point}
Suppose that $X_1,\dots,X_n$ are affine and 
put $X' := X_1 \times_{\FF_p} \cdots \times_{\FF_p} X_{n-1}$.
Let $U$ be an open dense subscheme of $X_n$, let $Q_n$ be the open formal subscheme of $P_n$ supported on $U_n$,
and let $Q$ be the open formal subscheme of $P$ supported on
$X' \times_{\FF_p} U_n$.
Let $M$ be a finite projective $\Gamma(P,\calO)[p^{-1}]$-module.
Then an element $\bv$ of $M \otimes_{\Gamma(P,\calO)} \Gamma(Q, \calO)$ belongs to $M$ if and only if
for every geometric point $\overline{x} \to X'$, the image of $\bv$ in
$M \otimes_{\Gamma(P, \calO)} \Gamma(W(\overline{x}) \times_{\ZZ_p} Q_n, \calO)$ belongs to
$M \otimes_{\Gamma(P, \calO)} \Gamma(W(\overline{x}) \times_{\ZZ_p} P_n, \calO)$.
\end{lemma}
\begin{proof}
By writing $M$ as a direct summand of a finite free module, we reduce to the case $M = \Gamma(P, \calO)[p^{-1}]$.
We may then check the claim modulo successive powers of $p$, which reduces to the corresponding mod-$p$ statement:
an element $f$ of $\Gamma(X' \times_{\FF_p} U_n, \calO) = \Gamma(X', \calO) \otimes_{\FF_p} \Gamma(U_n, \calO)$ belongs to $\Gamma(X, \calO)
= \Gamma(X', \calO) \otimes_{\FF_p} \Gamma(X_n, \calO)$ if and only if 
for every geometric point $\overline{x} \to X'$, the image of $f$ in $\kappa(\overline{x}) \otimes_{\FF_p} \Gamma(U, \calO)$
belongs to $\kappa(\overline{x}) \otimes_{\FF_p} \Gamma(X, \calO)$.
This is an instance of Lemma~\ref{L:mod p restrict to point}.
\end{proof}

The following is analogous to \cite[Theorem~5.3]{kedlaya-isocrystals}.
\begin{lemma} \label{L:fully faithful}
For $i=1,\dots,n$, let $U_i \subseteq X_i$ be an open immersion with dense image and put
$U := U_1 \times_{\FF_p} \cdots \times_{\FF_p} U_n$. Then
the restriction functor $\PhiaIsoc(X) \to \PhiaIsoc(U)$ is fully faithful.
\end{lemma}
\begin{proof}
It is sufficient to treat the case where $U_i = X_i$ for $i=1,\dots,n-1$, as the functor in question
is a composition of functors of this form.
We may assume that the $X_i$ are affine.
Using Remark~\ref{R:reduce to first power}, we reduce to the case $a=1$.
Using internal Homs, we further reduce to checking that for $\calE \in \PhiIsoc(X)$, 
the map $H^0(X, \calE) \to H^0(U, \calE)$ is surjective.
To check that a given element of $H^0(U, \calE)$ extends to $X$, 
by Lemma~\ref{L:restrict to point} it suffices to check this
after pullback from $X_1 \times_{\FF_p} \cdots \times_{\FF_p} X_{n-1}$ to an arbitrary geometric point;
we may thus reduce directly to \cite[Theorem~5.3]{kedlaya-isocrystals}.
\end{proof}

\begin{remark}
One can also give a proof of Lemma~\ref{L:fully faithful} that uses only the diagonal Frobenius, together with arc-descent as in the proof of 
Lemma~\ref{L:katz-filtration} below. We will record the details elsewhere; the proof above is needed to serve as a model for 
Lemma~\ref{L:fully faithful2} later.
\end{remark}

\section{External products}

We next introduce the fundamental construction of convergent $\Phi$-isocrystals via external products, and illustrate some examples of objects that do not arise in this fashion.

\begin{defn}
For $S = \{i_1 < \cdots < i_m\}$ a subset of $\{1,\dots,n\}$,
we may define a \emph{pullback functor} $\PhiaIsoc(X_{i_1} \times_{\FF_p} \cdots \times_{\FF_p} X_{i_m}) \to \PhiaIsoc(X)$.
To simplify notation, we only write this out in the case where $S = \{1,\dots,n-1\}$. 
Put $X' := X_1 \times_{\FF_p} \cdots \times_{\FF_p} X_{n-1}$, work locally as in Definition~\ref{D:phi-isocrystal},
and put $P' := P_1 \times_{\ZZ_p} \cdots \times_{\ZZ_p} P_{n-1}$.
Given $\calE' \in \PhiaIsoc(X')$, define its pullback to be the module $\calE' \otimes_{\Gamma(P', \calO)} \Gamma(P, \calO)$
and define the action of $\sigma_n$ and differentiation with respect to vector fields on $P_n$ so as to act trivially on $\calE'$.
\end{defn}

\begin{defn} \label{D:external product}
Let 
\[
S' = \{i_1 < \cdots < i_{m'}\} \sqcup S'' = \{j_1 < \cdots < j_{m''}\}
\]
be a partition of $\{1,\dots,n\}$ and put
\[
X' = X_{i_1} \times_{\FF_p} \cdots \times_{\FF_p} X_{i_{m'}}, \qquad
X'' = X_{j_1} \times_{\FF_p} \cdots \times_{\FF_p} X_{j_{m''}}.
\]
For $\calE' \in \PhiaIsoc(X'), \calE'' \in \PhiaIsoc(X'')$,
define the \emph{external product} (or \emph{box product}) $\calE' \boxtimes \calE'' \in \PhiaIsoc(X)$
as the tensor product $\tilde{\calE}' \otimes \tilde{\calE}''$ where
$\tilde{\calE}', \tilde{\calE}''$ are the respective pullbacks of $\calE', \calE''$.
\end{defn}

\begin{lemma} \label{L:same product}
With notation as in Definition~\ref{D:external product},
take $\calE'_1, \calE'_2 \in \PhiaIsoc(X')$ and
$\calE''_1, \calE''_2 \in \PhiaIsoc(X'')$,
all absolutely irreducible. If
$\Hom(\calE'_1 \boxtimes \calE''_1, \calE'_2 \boxtimes \calE''_2) \neq 0$, then $\calE'_1 \cong \calE'_2, \calE''_1 \cong \calE''_2$.
\end{lemma}
\begin{proof}
By symmetry it suffices to check that $H^0(X', \calE^{\prime \dual}_1 \otimes \calE'_2) \neq 0$, which we may do after enlarging $k_1,\dots,k_n$.
We start with a nonzero element of $H^0(X, (\calE^{\prime \dual}_1 \otimes \calE'_2) \boxtimes (\calE^{\prime \prime \dual}_1 \otimes \calE''_2))$.
After enlarging $k_1,\dots,k_n$ suitably, we may pull back from $X''$ to a geometric point to obtain a nonzero element of
$H^0(X', \calE^{\prime \dual}_1 \otimes \calE'_2)$.
\end{proof}

Many of the following results are motivated by the principle that ``all objects on $\PhiaIsoc(X)$ are derived from external products.'' However, this principle is subject to qualifications illustrated by the following two examples
(compare Lemma~\ref{L:external product decomposition} and Definition~\ref{D:absolutely irreducible}).

\begin{example}
Take $n=2$ and $X_i := \Spec k_i$.  Let $\calE \in \PhiIsoc(X)$
be free of rank 2 on the basis $\be_1, \be_2$ with the actions of $\varphi_1, \varphi_2$ being given by
\begin{gather*}
\varphi_1(\be_1) = \be_1, \quad \varphi_1(\be_2) = p \be_2, \quad \varphi_2(\be_1) = \be_1, \quad \varphi_2(\be_2) = p \be_2,
\end{gather*}
Then $\calE$ cannot be factored as $\calE_1 \boxtimes \calE_2$ with $\calE_i \in \FIsoc(X_i)$; however, it can be written as an extension of two objects of this form.
\end{example}

\begin{example}
Take $n = 2$, suppose that $k_i$ is algebraically closed, and take $X_i := \Spec k_i$. Let $\calE \in \PhiIsoc(X)$
be free of rank 2 on the basis $\be_1, \be_2$ with the actions of $\varphi_1$ and $\varphi_2$ being given by
\[
\varphi_1(\be_1) = \be_2, \quad \varphi_1(\be_2) = p \be_1, \quad \varphi_2(\be_1) = \be_2, \quad \varphi_2(\be_2) = p \be_1.
\]
Then $\calE$ cannot be factored as $\calE_1 \boxtimes \calE_2$ with $\calE_i \in \FIsoc(X_i)$. 
For $a$ even,
the image of $\calE$ in $\PhiaIsoc(X)$ does factor as $\calE_1 \boxtimes \calE_2$ with $\calE_i \in \FaIsoc(X_i)$,
but \emph{not uniquely}; contrast with Lemma~\ref{L:same product}.
\end{example}

\section{Stratification by the diagonal Newton polygon}
\label{sec:diagonal NP}

We next use the natural functor $\PhiaIsoc(X) \to \FaIsoc(X)$, obtained by forgetting the partial Frobenius actions and retaining only their composition, to obtain some structural results about $\Phi$-isocrystals.
We start with a form of the Grothendieck--Katz specialization theorem \cite[Theorem~2.3.1]{katz-slope}.

\begin{defn}
For $S$ a perfect ring of characteristic $p$, let $\FIsoc(S)$
be the category of finite projective $W(S)[p^{-1}]$-modules equipped with  isomorphisms with their $\varphi^a$-pullback.
We may then globalize to define $\FIsoc(X)$ whenever $X$ is a perfect $\FF_p$-scheme.

For $M \in \FIsoc(X)$, define the Newton polygon as a function of $x \in \Spec X$ by viewing the pullback of $M$ to $x$ as an object of $\FaIsoc(x)$ and taking its Newton polygon.
\end{defn}

\begin{lemma} \label{L:grothendieck-katz}
For $S$ a perfect scheme of characteristic $p$ and $M \in \FIsoc(S)$, the Newton polygon function of $M$ takes finitely many values, and each level set is a locally closed subspace of $\Spec S$
whose Zariski closure is the zero set of some finitely generated ideal of $S$.
\end{lemma}
\begin{proof}
We may assume at once that $M$ has constant rank and (as per Remark~\ref{R:reduce to first power}) that $a=1$.
Fix a finitely generated $W(S)$-submodule $M_0$ of $M$ which generates $M$ over $W(S)[p^{-1}]$.
For some nonnegative integer $n$, the action of $p^n \varphi$ carries $M_0$ into $M_0$; for any such $n$,
the Newton slopes are bounded below by $-n$. Similarly, for some $n$, the action of $p^n \varphi^{-1}$ carries $M_0$ into $M_0$; for any such $n$, the Newton slopes are bounded above by $n$. Since the denominators of these slopes
are bounded by $\rank(M)!$ (because the Newton polygon at every point has integral vertices), 
there are only finitely many possible values for the Newton polygon.

With this, it suffices to check that for each $\lambda \in \QQ$, the set of $x \in X$ at which the Newton polygon has all slopes $\geq \lambda$ is the zero set of some finitely generated ideal. (Namely, one may then apply this logic to the
exterior powers of $M$.) In the case where $M$ is free, we may take $M_0$ to be free
and then apply \cite[Theorem~2.3.1]{katz-slope}.
In the general case, we choose a surjection $F \to M$ with $F$ a free module and an associated splitting
$F \cong M \oplus M'$; equip $F$ with the trivial action of $\varphi$ (determined by some basis) multiplied by $p^n$
(for $n$ suitably large); equip $M'$ with the action given by $M' \to F \stackrel{\varphi}{\to} F \to M'$;
then apply the previous case to $M \oplus M'$. By taking $n$ suitably large, we ensure that the set we compute
is determined solely by the action of $\varphi$ on $M$.
\end{proof}

We will apply this via the following lemma. Recall that the notation of Definition~\ref{D:product notation} remains in effect.
\begin{lemma} \label{L:completed perfect closure}
Suppose that $X_i$ is affine for $i=1,\dots,n$.
Then the completed direct limit of $\Gamma(P, \calO)$ along $\sigma := \sigma_1 \circ \cdots \circ \sigma_n$
is isomorphic to $W(S)$ for $S := \Gamma(X^{\perf}, \calO)$.
\end{lemma}
\begin{proof}
The completed direct limit is a strict $p$-ring with residue ring $S$; the identification
then follows from the universal property of Witt vector rings.
\end{proof}

\begin{remark}
Although we will not use this here, we point out a refinement of Lemma~\ref{L:completed perfect closure}: the map $\Gamma(P, \calO) \to W(S)$ is module-split. This implies that it is also \emph{pure} (meaning that for every $\Gamma(P, \calO)$-module $M$, the map $M \to M \otimes_{\Gamma(P, \calO)} W(S)$ is injective), and hence an effective descent morphism for modules
\stacktag{08XA}.

To wit, each $\sigma_i$ is module-split by the easy direction of Kunz's regularity criterion,
as then is $\sigma\colon \Gamma(P, \calO) \to \Gamma(P, \calO)$. Fixing a module splitting $\tau$,
we may then view $\tau^n$ as a module splitting of $\sigma^n$ for all $n$.
Taking these together, completing, and using the identification from Lemma~\ref{L:completed perfect closure} yields a module-splitting of $\Gamma(P, \calO) \to W(S)$.
\end{remark}

For the remainder of \S\ref{sec:diagonal NP}, fix $\calE \in \PhiaIsoc(X)$.

\begin{defn}
For $x \to X$ a point, we may pull back $\calE$ to $\FaIsoc(x)$
by retaining only the action of the \emph{diagonal Frobenius} $\varphi^a := \varphi_1^a \circ \cdots \circ \varphi_n^a$.
In particular, we may associate a \emph{diagonal Newton polygon} to this pullback.
\end{defn}

\begin{theorem} \label{T:total Newton stratification}
For $x \to X$ a point, the diagonal Newton polygon of $\calE$ at $x$
depends only on the images of $x$ in $X_1,\dots,X_n$. Moreover, the level sets of the diagonal Newton polygon 
form a locally closed stratification of $|X_1| \times \cdots \times |X_n|$.
\end{theorem}
\begin{proof}
We may assume that $X_i$ is affine and set notation as in Definition~\ref{D:phi-isocrystal} and Lemma~\ref{L:completed perfect closure}.
By pulling back to $W(S)$, we may apply Lemma~\ref{L:grothendieck-katz} to obtain a locally closed stratification of $X$. As this stratification must be invariant under the action of $\varphi_1^a, \dots, \varphi_n^a$, we may apply Theorem~\ref{T:stable open} to deduce the desired result.
\end{proof}

\begin{defn}
We say that $\calE$ is \emph{diagonally unit-root} if the diagonal Newton polygon of $\calE$ is constant as a function of $|X|$
with all slopes equal to $0$.
\end{defn}

\begin{lemma} \label{L:spectrum valuation ring}
For any perfect scheme $S$ of characteristic $p$ and any positive integer $a$, the formula $V \mapsto V \otimes_{\underline{\ZZ_{p^a}}} W(\calO)[p^{-1}]$ defines an equivalence of categories between $\QQ_{p^a}$-local systems (interpreted here as finite locally free modules over the pro-\'etale locally constant sheaf $\underline{\QQ_{p^a}}$) 
and the full subcategory of unit-root objects of $\FaIsoc(S)$.
\end{lemma}
\begin{proof}
Suppose first that $S$ is the spectrum of a strictly henselian perfect valuation ring $R$ of characteristic $p$. 
Let $\eta$ and $s$ be the generic point and the closed point of $S$. Given an object $M$ of $\FaIsoc(S)$,
its pullback to $\FaIsoc(\eta)$ is the base extension of $M$ to $W(\Frac(R))[p^{-1}]$. If we start with a unit-root object, then
we may choose a lattice in the base extension stable under $\varphi^a$ and $\varphi^{-a}$ 
(by starting with any lattice and taking the span of its images under $\varphi^{a\ZZ}$, using \cite[Basic Slope Estimate 1.4.3]{katz-slope} to see that this span is bounded).
We may then invoke \cite[Theorem~2.7]{kedlaya-ainf} to obtain a finite free $W(R)$-module equipped with a semilinear $\varphi^a$-action;
this then defines an object of $\FaIsoc(S)$ which is the pullback of a unit-root object of $\FaIsoc(s)$.
Since $s$ is a geometric point, this proves the claim.

To deduce the general case, we first introduce the \emph{arc-topology} on the category of perfect affine $\FF_p$-schemes in the sense of \cite{bhatt-mathew}
(previously considered in \cite{rydh}).
For this topology, one has effective descent for the stack of finite projective $W(\calO)[p^{-1}]$-modules \cite[Proposition~5.9]{ivanov}. Using \cite[Example~1.3 and Proposition~3.30]{bhatt-mathew}, we see that every perfect affine $\FF_p$-scheme admits an arc-covering by a product of AIC (absolutely integrally closed) height-1 valuation rings.

By applying arc-descent once, we see that the formula in question does indeed define a functor valued in $\FaIsoc(S)$; it is then apparent that the essential image contains only unit-root objects. To obtain the claimed equivalence, we again use arc-descent to reduce to the case where $S$ is the spectrum of a product of AIC height-1 valuation rings; we then formally reduce to the case where $S$ is the product of a single AIC height-1 valuation ring (as in the proof of \cite[Theorem~6.1]{ivanov}). We then apply the first paragraph to conclude.
\end{proof}

\begin{lemma} \label{L:diagonal fully faithful}
Put $S := \Gamma(X^{\perf}, \calO)$.
Let $M$ be a finite projective $\Gamma(P, \calO)[p^{-1}]$-module
equipped with an isomorphism with its $\varphi^a$-pullback.
Suppose that $\tilde{M} := M \otimes_{\Gamma(P, \calO)} W(S)$ admits
a finite projective $W(S)$-submodule $\tilde{M}_0$ such that 
$\varphi^a$ carries $\tilde{M}_0$ into itself and
the natural map $\tilde{M}_0 \otimes_{W(S)} W(S)[p^{-1}] \to \tilde{M}$ is an isomorphism.
Then every $\bv \in \tilde{M}$ with $(\varphi^a-1)(\bv) \in M$ itself belongs to $M$.
\end{lemma}
\begin{proof}
It will suffice to check the claim for $\bv \in \tilde{M}_0$.
It will further suffice to show that $\bv = \bv_0 + p\bv_1$ for some $\bv_0 \in M, \bv_1 \in \tilde{M}_0$.
Namely, we can then iterate to show that $\bv$ is a convergent (in $\tilde{M}$) sum of elements of $M$, and hence itself belongs to $M$.

To prove the claim, note that \emph{a priori} we have a similar representation where $\bv_0 \in \varphi^{-an}(M)$ for some nonnegative integer $n$; it will further suffice to show that if $n > 0$, then the same holds with $n$ replaced by $n-1$. This follows by writing
\[
\bv_0 =  \varphi^a(\bv_0) - (\varphi^a-1)(\bv) + p(\varphi^a(\bv_1) - \bv_1). \qedhere
\]
\end{proof}

\begin{theorem} \label{T:unit-root representation}
Suppose that for $i=1,\dots,n$, $k_i$ is algebraically closed and $X_i$ is connected.
Fix a geometric point $\overline{x}$ of $X$ and let $X_0$ be the connected component of $X$ over which $\overline{x}$ lies.
Then the following three categories are equivalent.
\begin{enumerate}
\item[(a)]
The full category of diagonally unit-root objects of $\PhiaIsoc(X)$.
\item[(b)]
The full category of unit-root objects of $\FaIsoc(X^{\perf})$ equipped with
commuting semilinear actions of $\varphi_1^a,\dots,\varphi_n^a$ which compose to the action of $\varphi^a$ (i.e., the ``diagonally unit-root objects of $\PhiaIsoc(X^{\perf})$'').
\item[(c)]
The category of $\QQ_{p^a}$-local systems on $X$ (or equivalently $X^{\perf}$)
equipped with commuting actions of $\varphi_1^a,\dots,\varphi_n^a$ which compose to the canonical action of $\varphi^a$.
\item[(d)]
The category of
continuous representations of $G := \pi_1(X_0, \overline{x})$
on finite-dimensional $\QQ_{p^a}$-vector spaces $V$ equipped with 
$n$ commuting $G$-equivariant automorphisms which compose to the identity.
\end{enumerate}
\end{theorem}
\begin{proof}
By  Lemma~\ref{L:spectrum valuation ring}, the category of unit-root objects of 
$\FaIsoc(X^{\perf})$ is equivalent to the category of $\QQ_{p^a}$-local systems on $X^{\perf}$. This yields the equivalence between (b) and (c).

By Corollary~\ref{C:geometrically integral product}, $\pi_0(X)$ is a torsor for 
the group $\coker(\Delta\colon\widehat{\ZZ} \to \widehat{\ZZ}^n) \cong \widehat{\ZZ}^{n-1}$,
and a $\QQ_{p^a}$-local system on $X$ can be described
concretely as a finite projective module $M$ over the ring $R := \Cts(\pi_0(X), \QQ_{p^a})$
equipped with a continuous homomorphism $\rho\colon G \to \Aut_R(M)$.
An action of the $i$-th partial Frobenius on a unit-root $F^a$-isocrystal on $X$
then corresponds to an isomorphism of the corresponding module $M$ with its pullback under the ring automorphism of $R$ induced by translation by the $i$-th generator of $\widehat{\ZZ}^n$.
In particular, an object of (c)  corresponds to a finite projective $R$-module $M$ (now necessarily of constant rank) equipped with a continuous $G$-action plus commuting $G$-equivariant isomorphisms
with its pullbacks by the translation automorphisms which compose to the identity.
With this description, there is an evident fully faithful base extension functor from (d) to (c); 
this functor is also essentially surjective.
This yields the equivalence between (c) and (d).

As a corollary of the equivalence between (b) and (d), we deduce that
for $X$ affine and $S := \Gamma(X^{\perf}, \calO)$, every object $M$ of $\PhiaIsoc(X^{\perf})$
admits a $\varphi^a$-stable lattice, meaning a finite projective $W(S)$-submodule  $M_0$ such that the induced map $M_0 \otimes_{W(S)} W(S)[p^{-1}] \to M$ is an isomorphism and the action of $\varphi^a$ on $M$ induces an isomorphism $\varphi^a M_0 \cong M_0$. (Note however that $M_0$ is not necessarily stable under the actions of the individual $\varphi_i^a$.)

Using the previous paragraph together with Lemma~\ref{L:diagonal fully faithful}, we deduce that the base extension functor from (a) to (b) is fully faithful (whether or not $X$ is affine). To establish that this functor is also essentially surjective, it suffices to factor the equivalence from (d) to (b) through (a);
for this we apply \cite[Proposition~4.1.1]{katz-modular}
as in the proof of \cite[Theorem~2.1]{crew-f}.
\end{proof}

We next adapt Katz's construction of slope filtrations \cite[Corollary~2.6.2]{katz-slope}.
\begin{lemma} \label{L:katz-filtration}
Let $S$ be a perfect ring of characteristic $p$.
Let $M$ be a finite projective $W(S)[p^{-1}]$-module equipped with an isomorphism with its $\varphi^a$-pullback.
Suppose that the function assigning to $x \in \Spec S$ the Newton polygon of the pullback of $M$ to $x$ is constant.
Then there exist a uniquely split filtration
\[
0 = M_0 \subset \cdots \subset M_l = M
\]
of $M$ by $\varphi^a$-stable submodules and an ascending sequence $\mu_1 < \cdots < \mu_l$ of rational numbers
such that for $j=1,\dots,l$, $M_j/M_{j-1}$ is a finite projective $W(S)[p^{-1}]$-module and the Newton polygon of its pullback to any $x \in \Spec S$ has all slopes equal to $\mu_j$. Moreover, the filtration and sequence are both uniquely determined by this requirement. (We call this the \emph{slope filtration} of $M$.)
\end{lemma}
\begin{proof}
Using arc-descent as in the proof of Lemma~\ref{L:spectrum valuation ring}, we reduce to the case where $S$ is an AIC height-1 valuation ring.
By \cite[Theorem~2.7]{kedlaya-ainf}, any finite projective $W(S)[p^{-1}]$-module is free; moreover, we can choose a basis for which the Hodge polygon of $\varphi^a$ coincides with the Newton polygon. We may then apply \cite[Theorem~2.4.2]{katz-slope} to obtain the desired filtration and \cite[Theorem~2.5.1]{katz-slope} to obtain the unique splitting.
\end{proof}

\begin{theorem} \label{T:total slope filtration}
Suppose that the diagonal Newton polygon of $\calE$ is constant as a function on $|X|$. (By Theorem~\ref{T:total Newton stratification}, this function always factors through $|X_1| \times \cdots \times |X_n|$.)
Then there exist a filtration
\[
0 = \calE_0 \subset \cdots \subset \calE_l = \calE
\]
in $\PhiaIsoc(X)$ and an ascending sequence $\mu_1 < \cdots < \mu_l$ of rational numbers such that for $j=1,\dots,l$,
the diagonal Newton polygon of $\calE_j/\calE_{j-1}$ is everywhere equal to a polygon with all slopes $\mu_j$. Moreover, the filtration and sequence are both uniquely determined by this requirement. (We call this the \emph{diagonal slope filtration} of $\calE$.)
\end{theorem}
\begin{proof}
We may assume at once that $X$ is affine.
We may pull back $\calE$ to obtain an object $\tilde{\calE} \in \FaIsoc(X^{\perf})$. Apply Lemma~\ref{L:katz-filtration}
and let $\tilde{\calE}_1$ be the first nonzero step of the resulting filtration of $\tilde{\calE}$; by the uniqueness of the filtration, $\tilde{\calE}_1$ is stable under the actions of the $\varphi_i^a$. 
It will
suffice to show that $\tilde{\calE}_1$ is itself the base extension of a subobject $\calE_1$ of $\calE$ in $\PhiaIsoc(X)$;
for this purpose, we may increase $a$ and then twist to reduce to the case
$\mu_1 = 0$. We may then apply Theorem~\ref{T:unit-root representation} to produce a diagonally unit-root object $\calE_1 \in \PhiaIsoc(X)$ which pulls back to $\tilde{\calE}_1$; it then remains to check that the inclusion $\tilde{\calE}_1  \to \tilde{\calE}$ descends to a morphism $\calE_1 \to \calE$. This follows by applying Lemma~\ref{L:diagonal fully faithful} to $\calE_1^\dual \otimes \calE$.
\end{proof}

\begin{remark}
In Lemma~\ref{L:katz-filtration}, using the fact that $S$ is perfect, one can similarly apply \cite[Proposition~5.4.3]{kedlaya-slope-filt} to show that the filtration splits uniquely; that is, the slope filtration refines to a \emph{slope decomposition}. However, this is not the case in Theorem~\ref{T:total slope filtration}.
\end{remark}

\section{Relative Dieudonn\'e--Manin}
\label{sec:relative DM}

In the case where one of the factors in the product is a geometric point, we obtain a relative version of the usual Dieudonn\'e--Manin decomposition theorem \cite[Corollary~14.6.4]{kedlaya-book}.

\begin{hypothesis}
Throughout \S\ref{sec:relative DM}, 
suppose that $k_n$ is algebraically closed and $X_n = \Spec k_n$,
and put $X' := X_1 \times_{\FF_p} \cdots \times_{\FF_p} X_{n-1}$.
Let $\pi_n\colon X \to X_n$ denote the projection map.
\end{hypothesis}

\begin{lemma} \label{L:cohomology pullback}
Assume that $X$ is affine.
For $\calE' \in \PhiaIsoc(X')$ and $i \geq 0$,
let $\calE$ be the pullback of $\calE'$ from $X'$ to $X$; then
\[
H^i_{\rig}(X, \calE)^{\varphi_n} \cong H^i_{\rig}(X', \calE'),
\qquad
 H^i_{\rig}(X, \calE)_{\varphi_n} = 0.
\]
\end{lemma}
\begin{proof}
We may assume that $k$ is at most countably generated over $\FF_p$.
With notation as in Definition~\ref{D:phi-isocrystal}, put $P' := P_1 \times_{\ZZ_p} \cdots \times_{\ZZ_p} P_{n-1}$.
Using a Schauder basis for $K_n$ over $\QQ_p$ \cite[Lemma~1.3.11]{kedlaya-book}, we see that for each $i$,
\[
\Gamma(P, \Omega^i \otimes \calE)^{\varphi_n} = \Gamma(P, \Omega^i \otimes \calE), \qquad \Gamma(P, \Omega^i \otimes \calE)_{\varphi_n} = 0.
\]
We may then deduce the claim by comparing the two spectral sequences that compute the total cohomology of
\[
\Omega^\bullet \otimes \calE \stackrel{\varphi_n-1}{\to} \Omega^\bullet \otimes \calE. \qedhere
\]
\end{proof}

\begin{theorem} \label{T:relative DM}
Every object $\calE \in \PhiaIsoc(X)$ decomposes uniquely as a direct sum
\[
\calE \cong \bigoplus_{d \in \QQ} \calE_d
\]
in which for $d = \frac{r}{s}$ in lowest terms,
$\calE_d$ is obtained by pulling back an object of $\PhiaIsoc(X')$ equipped with an
endomorphism $F_n$ such that $F_n^s = p^r$, which then gives the action of $\varphi_n$ on the pullback.
 (The latter may be recovered from $\calE_d$
as the kernel of $\varphi_n^{as} - p^r$.)
\end{theorem}
\begin{proof}
We may assume that $X$ is affine.
Suppose first that the diagonal Newton polygon of $\calE$ is constant as a function on $|X'|$.
Using Lemma~\ref{L:cohomology pullback} to compare extensions, we may reduce to the case where $\calE$ is irreducible.
By Theorem~\ref{T:total slope filtration}, this means that the diagonal Newton slopes of $\calE$ are all equal to a single value.
By increasing $a$ and then twisting, we may reduce to the case where this single value is 0.
In this case, after enlarging $k_1,\dots,k_{n-1}$ to become algebraically closed,
Theorem~\ref{T:unit-root representation} gives rise to a finite-dimensional $\QQ_{p}$-vector space $V$
admitting an invertible linear action of $\varphi_n$. 
If we tensor this vector space over $\QQ_{p}$ with $K_n$, then we may apply the usual Dieudonn\'e--Manin
theorem \cite{manin} (see also \cite[Corollary~14.6.4]{kedlaya-book})
to decompose into eigenspaces as specified.

Now consider the general case. By Theorem~\ref{T:total Newton stratification}, there exists an open dense subspace
$U'$ of $X'$ of the form $U'_1 \times_{\FF_p} \cdots \times_{\FF_p} U'_{n-1}$ on which the diagonal Newton polygon is constant.
By Lemma~\ref{L:fully faithful}, we may extend projectors defining the direct sum decomposition from $U' \times_{\FF_p} k_n$
to $X' \times_{\FF_p} k_n = X$; we may thus reduce to the previous paragraph.
\end{proof}

\begin{cor} \label{C:relative DM}
Suppose that $k_1,\dots,k_n$ are algebraically closed and that $X_i = \Spec k_i$ for $i=1,\dots,n$. 
Then every object of $\PhiaIsoc(X)$ decomposes uniquely as a direct sum
\[
\bigoplus_{d_1,\dots,d_n \in \QQ} \calE_{d_1,\dots,d_n}
\]
in which for $d_1,\dots,d_n \in \QQ$ with least common denominator $s$, 
$\calE_{d_1,\dots,d_n}$ is obtained by pulling back a  finite-dimensional $\QQ_p$-vector space equipped
with commuting endomorphisms $F_1,\dots,F_n$ such that $F_i^s = p^{d_i s}$, which then give the actions
of $\varphi_1^a,\dots,\varphi_n^a$ on the pullback.  (This vector space may be recovered from $\calE_{d_1,\dots,d_n}$
as the joint kernel of $\varphi_i^{as} - p^{d_i s}$ for $i=1,\dots,n$.)
\end{cor}
\begin{proof}
This follows by repeated application of Theorem~\ref{T:relative DM}.
\end{proof}

\section{Joint Newton slopes}
\label{sec:joint NP}

As a corollary of the relative Dieudonn\'e--Manin theorem, we may now define and study a ``Newton polygon'' that keeps track of the slopes for each of the partial Frobenius maps separately.
Throughout \S\ref{sec:joint NP}, fix $\calE \in \PhiaIsoc(X)$.

\begin{defn}
Define \emph{joint Newton slopes} of $\calE$ at a point $(x_1,\dots,x_n) \in |X_1| \times \cdots \times |X_n|$ as follows.
Choose geometric points $\overline{x}_i$ lying over $x_i$, put $\overline{X} := \overline{x}_1 \times_{\FF_p} \cdots \times_{\FF_p} \overline{x}_n$,
and apply Corollary~\ref{C:relative DM} to the pullback of $\calE$ to $\PhiaIsoc(\overline{X})$.
For $(d_1,\dots,d_n) \in \QQ^n$, we declare the \emph{multiplicity} of $(d_1,\dots,d_n)$ as a slope of $\calE$ to be
$\rank(\calE_{d_1,\dots,d_n})$. This does not depend on the choice of the $\overline{x}_i$.
\end{defn}

We have the following refinement of Theorem~\ref{T:total Newton stratification}.
\begin{theorem} \label{T:joint Newton stratification}
The joint Newton slope multiset of $\calE$, as a function on $|X|$,
is constant on each (connected) stratum of the stratification by the diagonal Newton polygon (Theorem~\ref{T:total Newton stratification}).
In particular, it also factors through $|X_1| \times \cdots \times |X_n|$.
\end{theorem}
\begin{proof}
We may first apply Theorem~\ref{T:total Newton stratification} to reduce to the case where the diagonal Newton polygon is constant
on $|X|$, then apply Theorem~\ref{T:total slope filtration} to further reduce to the case where $\calE$ is diagonally unit-root.
In this case, the claim follow at once from Theorem~\ref{T:unit-root representation}.
\end{proof}

\section{Partial slope filtrations}

Now that we have a meaningful notion of joint Newton slopes for a $\Phi$-isocrystal, we derive a form of the slope filtration theorem.

\begin{lemma} \label{L:split by slope}
Let $\calE_i, \calE'_i \in \FaIsoc(X_i)$ be objects which are isoclinic of respective slopes $\mu_i, \mu'_i$. Put
$\calE := \calE_1 \boxtimes \cdots \boxtimes \calE_n, \calE' := \calE'_1 \boxtimes \cdots \boxtimes \calE'_n$.
\begin{enumerate}
\item[(a)]
If $H^0(X, \calE^\dual \otimes \calE') \neq 0$, then
$\mu_i = \mu'_i$ for all $i$.
\item[(b)]
If $H^1(X, \calE^\dual \otimes \calE') \neq 0$, then $\mu_i \geq \mu'_i$ for all $i$.
\end{enumerate}
\end{lemma}
\begin{proof}
In both cases, we may reduce to the case where $\calE_i$ is the trivial object and $\mu_i = 0$ for all $i$; by permuting indices, we may reduce to proving the claimed 
assertion for $i=n$. 

To check (a), we may pull back from $X_1 \times_{\FF_p} \cdots \times_{\FF_p} X_{n-1}$ to a suitable geometric point to reduce to the case $n=1$.

To check (b), suppose first that $X = \Spec R$ is affine.
Suppose that $0 < \mu'_n$. Given an extension of $\calE'$,
after pulling back from $X_1 \times_{\FF_p} \cdots \times_{\FF_p} X_{n-1}$ to any geometric point we obtain a \emph{unique} splitting for the action of $\varphi_n$.
In particular, this splitting spreads out uniquely over nilpotent thickenings.
Since the homomorphism from $R$ to the product of its completions at maximal ideals is faithfully flat (by applying \stacktag{00MB} and \stacktag{05CZ}
as in \cite[Remark~6.4]{grubb-kedlaya-upton}) and the uniqueness gives us descent data for the resulting fpqc covering, we obtain a unique splitting of the original extension for the action of $\varphi_n$. The uniqueness implies that the splitting is also preserved by the other actions, yielding $H^1(X, \calE') = 0$.
Finally, the uniqueness of the splitting allows us to extend the argument to the case where $X$ is not necessarily affine.
\end{proof}

\begin{lemma} \label{L:diagonally unit-root decomposition}
For $d \in \ZZ$, let $\calO(d) \in \FaIsoc(X_i)$ be the object of rank $1$ with a generator $\bv$ satisfying $\varphi^a(\bv) = p^{d} \bv$.
Suppose that $\calE \in \PhiaIsoc(X)$ is diagonally unit-root.
Then after replacing $a$ by some suitable multiple, $\calE$ decomposes as a direct sum
\[
\bigoplus_{d_1,\dots,d_n} \calE_{d_1,\dots,d_n} 
\]
indexed by tuples $(d_1,\dots,d_n) \in \ZZ^n$ with $d_1 + \cdots + d_n = 0$, in which
$\calE_{d_1,\dots,d_n}$ is a successive extension of objects each of the form
\[
(\calF_1 \otimes \calO(d_1)) \boxtimes \cdots \boxtimes (\calF_n \otimes \calO(d_n))
\]
for some unit-root objects $\calF_i \in \FaIsoc(X_i)$.
\end{lemma}
\begin{proof}
Choose a geometric point $\overline{x}$ of $X$.
By Theorem~\ref{T:unit-root representation}, $\calE$ corresponds to a continuous representation of $G := \pi_1(X_1, \overline{x}) \times \cdots \times \pi_1(X_n, \overline{x})$
on a finite-dimensional $\QQ_{p^a}$-vector space $V$ equipped with $G$-equivariant, commuting, invertible actions of $\varphi_1^a, \dots, \varphi_n^a$
which compose to the identity. 
After replacing $a$ by a suitable multiple and enlarging $k_1,\dots,k_n$ to be algebraically closed, we may apply Corollary~\ref{C:relative DM} at $\overline{x}$ to deduce that for each $i$, the eigenvalues of $\varphi_i^a$ on $V$
belong to $\QQ_{p^a}$. By separating $V$ into its joint eigenspaces and then twisting, we may reduce to the case where $\varphi_1^a, \dots, \varphi_n^a$ all fix $V$.

In this case, we prove that $\calE \cong \calF_1 \boxtimes \cdots \boxtimes \calF_n$ by induction on $n$, as follows.
We may of course assume $V \neq 0$.
By further enlarging $a$, we may reduce to the case where the irreducible constitutents of $V$ as a representation of $\pi_1(X_n, \overline{x})$
are absolutely irreducible; let $W$ be such a constituent which occurs as a subobject of $V$.
By Theorem~\ref{T:unit-root representation}, $W$ corresponds to an irreducible unit-root object $\calF_n \in \FaIsoc(X_n)$,
while $\Hom_{\pi_1(X_n, \overline{x})}(W, V)$ with its action of $\pi_1(X_1, \overline{x}) \times \cdots \times \pi_1(X_{n-1}, \overline{x})$
corresponds to a diagonally unit-root object $\calE' \in \PhiaIsoc(X_1 \times_{\FF_p} \cdots \times_{\FF_p} X_{n-1})$
such that $\calE' \boxtimes \calF_n$ maps injectively to $\calE$. Applying the induction hypothesis to $\calE'$ proves the claim.
\end{proof}

\begin{cor} \label{C:partial slope filtration}
Suppose that $\calE \in \PhiaIsoc(X)$ has constant diagonal Newton polygon. Then for $i=1,\dots,n$, $\calE$ admits a filtration
\[
0 = \calE_0 \subset \cdots \subset \calE_l = \calE
\]
in $\PhiaIsoc(X)$ and an ascending sequence $\mu_1 < \cdots < \mu_l$ of rational numbers such that for $j=1,\dots,l$,
every joint Newton slope of $\calE_j/\calE_{j-1}$ has $i$-th component equal to $\mu_j$. Moreover, the filtration and sequence are both uniquely determined by this requirement. (We call this the \emph{$i$-th partial slope filtration} of $\calE$.)
\end{cor}
\begin{proof}
By Theorem~\ref{T:joint Newton stratification}, the joint Newton polygon of $\calE$ is also constant.
By Remark~\ref{R:reduce to first power}, we are free to enlarge $a$ to ensure that the joint Newton slopes are integers.
We may then apply Lemma~\ref{L:diagonally unit-root decomposition} to the successive quotients of the diagonal slope filtration;
this yields a filtration in which each successive quotient has the property that every joint Newton slope has a fixed value in the $i$-th component,
but these values may not occur in ascending order. However, using Lemma~\ref{L:split by slope} we can reorder the successive quotients to enforce this condition.
\end{proof}

\section{Slices}
\label{sec:slices}

We finally arrive at our main structure theorem for convergent $\Phi$-isocrystals,
modeled on Corollary~\ref{C:external product decomposition lisse}.

\begin{hypothesis}
For the remainder of \S\ref{sec:slices},
put $X' := X_1 \times_{\FF_p} \cdots \times_{\FF_p} X_{n-1}$
and choose  $\calE \in \PhiaIsoc(X)$.
Apply Theorem~\ref{T:total Newton stratification} to choose
open dense subschemes $U_i$ of $X_i$ such that the diagonal Newton polygon is constant on
$U := U_1 \times_{\FF_p} \cdots \times_{\FF_p} U_n$.
\end{hypothesis}

\begin{lemma} \label{L:pullback subquotient}
Let $\pi_n\colon X \to X_n$ be the projection map.
Then the essential image of the pullback functor $\pi_n^*$ from $\FaIsoc(X_n)$ to $\PhiaIsoc(X)$
 is closed under formation of quotients (and hence subquotients).
\end{lemma}
\begin{proof}
Choose $\calE \in \FaIsoc(X_n)$ and let $\pi_n^* \calE \to \calF$ be a surjection in $\PhiaIsoc(X)$.
To prove the claim, we may do so after enlarging $k_1,\dots,k_n$; we may thus assume that there exists
a $k_n$-rational point $x \in X'$. In this case, we obtain a candidate for an object which pulls back to $\calF$ by pulling the surjection
$\pi_n^* \calE \to \calF$ back to $x \times_k X_n \cong X_n$. To check that this candidate does indeed have pullback isomorphic to $\calF$,
by Lemma~\ref{L:fully faithful} we may reduce to the case where $U_i = X_i$ for $i=1,\dots,n$.
Using Theorem~\ref{T:total Newton stratification}, we may further reduce to the case where the diagonal Newton polygon is constant.
We may then deduce the claim from Theorem~\ref{T:unit-root representation}.
\end{proof}

\begin{lemma} \label{L:pushforward}
Let $\pi_n\colon X \to X_n$ be the projection map.
\begin{enumerate}
\item[(a)]
The pullback functor $\pi_n^*$ from $\FaIsoc(X_n)$ to $\PhiaIsoc(X)$ admits a right adjoint $\pi_{n*}$.
\item[(b)]
The formation of $\pi_{n*}(\calE)$ commutes with base change from $U_n$ to any other scheme.
\end{enumerate}
\end{lemma}
\begin{proof}
Note that $\pi_n^*$ is fully faithful
and by Lemma~\ref{L:pullback subquotient} its essential image is closed under formation of subquotients.
 Hence given $\calE \in \PhiaIsoc(X)$, the collection of subobjects of $\calE$ which belong to the essential image of $\pi_n^*$
has a maximal element $\calF$; we define $\pi_{n*}(\calE)$ to be the object of $\FaIsoc(X_n)$ of which $\calF$ is a pullback. This yields (a).

To prove (b), using Lemma~\ref{L:fully faithful} we may reduce to the case where $U_i = X_i$ for $i=1,\dots,n$.
Using Theorem~\ref{T:total slope filtration} (and suitably enlarging $a$), we may further reduce to the case where
$\calE$ is diagonally unit-root. In this case, the claim follows at once from
Theorem~\ref{T:unit-root representation}.
\end{proof}

\begin{defn}
For $\calE \in \PhiaIsoc(X)$ and $\overline{x}_n \to X_n$ a geometric point, 
we may apply Theorem~\ref{T:relative DM} to split the pullback of $\calE$ to $\PhiaIsoc(X' \times_{\FF_p} \overline{x}_n)$
as a direct sum $\bigoplus_{r \in \QQ} \calE_r$.
For $r \in \ZZ$, we may further write $\calE_r = \calE'_r \boxtimes \calO(r)$ for some
$\calE'_r \in \PhiaIsoc(X')$. The object $\calE'_r$ is called the \emph{$r$-slice} of $\calE$ over $\overline{x}_n$.
\end{defn}

\begin{lemma} \label{L:same slices}
For $r \in \ZZ$, the $r$-slices of $\calE$ over any two geometric points of $U_n$ are isomorphic.
\end{lemma}
\begin{proof}
By twisting, we may reduce to the case $r=0$.
By Lemma~\ref{L:fully faithful}, we may reduce to the case $U = X$.
By Corollary~\ref{C:partial slope filtration}, we may further reduce to the case where the $n$-th partial slope filtration of $\calE$ consists of a single step.
In this case, by analogy with Theorem~\ref{T:unit-root representation},
we may construct a continuous action of $\pi_1(X_n, \overline{x}_n)$ on the $r$-slice of $\calE$ over $\overline{x}_n$;
we then obtain isomorphisms of $r$-slices over distinct geometric points via parallel transport.
\end{proof}

\begin{lemma} \label{L:convergent decomposition}
Let $\calE'_r$ be the $r$-slice of $\calE$ at some geometric point of $U_n$
and suppose that this is nonzero.
Then there exist an object $\calG \in \FaIsoc(U_n)$ and a nonzero map $\calE'_r \boxtimes \calG \to \calE|_{X' \times_{\FF_p} U_n}$.
\end{lemma}
\begin{proof}
Let $\pi_n^{\prime *}$ denote the pullback functor from $\PhiaIsoc(X')$ to $\PhiaIsoc(X)$.
Put $\calG := \pi_{n*}(\pi_n^{\prime *}(\calE^{\prime \dual}_r) \otimes \calE)$; by Lemma~\ref{L:same slices} and Lemma~\ref{L:pushforward}(b), we have $\calG \neq 0$.
The desired map is given by composing the canonical inclusion $\calE'_r \boxtimes \calG \to \pi_n^{\prime *}(\calE'_r) \otimes (\pi_n^{\prime *}(\calE^{\prime \dual}_r) \otimes \calE)$,
the identification $\pi_n^{\prime *}(\calE'_r) \otimes (\pi_n^{\prime *}(\calE^{\prime \dual}_r) \otimes \calE) \cong
\pi_n^{\prime *}(\calE'_r \otimes \calE^{\prime \dual}_r) \otimes \calE$,
and the contraction $\pi_n^{\prime *}(\calE'_r \otimes \calE^{\prime \dual}_r) \otimes \calE \to \calE$. We may see that this map is nonzero by slicing again.
\end{proof}

\begin{lemma} \label{L:convergent decomposition2}
Suppose that $\calE$ is absolutely irreducible. Then 
there exist absolutely irreducible objects $\calF \in \PhiaIsoc(X'), \calG \in \FaIsoc(X_n)$ such that $\calE \cong \calF \boxtimes \calG$.
\end{lemma}
\begin{proof}
As per Remark~\ref{R:reduce to first power}, we are free to increase $a$.
By Lemma~\ref{L:convergent decomposition}, we have absolutely irreducible objects $\calF \in \PhiaIsoc(X'), \calG \in \FaIsoc(U_n)$ 
and a surjective morphism $\calF \boxtimes \calG \to \calE|_{X' \times_{\FF_p} U_n}$.
Pulling back from $X'$ to a product of geometric points $\overline{x}'$,
we obtain a surjective morphism $\calF|_{\overline{x}'} \boxtimes \calG \to \calE|_{\overline{x}' \times_{\FF_p} U_n}$.
By decomposing $\calE|_{\overline{x}' \times_{\FF_p} X_n}$ according to Theorem~\ref{T:relative DM},
we obtain a surjective morphism $\calG^{\oplus m} \to \calH|_{U_n}$ in $\FaIsoc(U_n)$
for some nonzero $\calH \in \FaIsoc(X_n)$ and some positive integer $m$. In particular, all of the constituents of $\calH$ in $\FaIsoc(U_n)$
are isomorphic to $\calG$, so there exists a surjection $\calH|_{U_n} \to \calG$. 
We may thus replace $\calG$ with an absolutely irreducible quotient $\calG'$ of $\calH$ in $\FaIsoc(X_n)$,
then apply Lemma~\ref{L:fully faithful} to extend the resulting morphism $\calF \boxtimes (\calG'|_{U_n}) \to \calE|_{X' \times_{\FF_p} U_n}$ to a surjective morphism $\calF \boxtimes \calG' \to \calE$.
Since $\calF$ and $\calG'$ are both absolutely irreducible, so is 
$\calF \boxtimes \calG'$ because
\[
H^0(X, (\calF \boxtimes \calG')^\dual \otimes (\calF \boxtimes \calG'))
\cong H^0(X', \calF^\dual \otimes \calF) \otimes_{\QQ_{p^a}}
H^0(X_n, \calG^{\prime \dual} \otimes \calG)
\cong \QQ_{p^a} \otimes_{\QQ_{p^a}} \QQ_{p^a} \cong \QQ_{p^a};
\]
the map $\calF \boxtimes \calG' \to \calE$ must therefore be an isomorphism.
\end{proof}

\begin{theorem} \label{T:convergent decomposition}
Any absolutely irreducible object in $\PhiaIsoc(X)$ is isomorphic to
$\calE_1 \boxtimes \dots \boxtimes \calE_n$ for some absolutely irreducible objects $\calE_i \in \FaIsoc(X_i)$.
\end{theorem}
\begin{proof}
This follows from Lemma~\ref{L:convergent decomposition2} by induction on $n$.
\end{proof}

\begin{remark} \label{R:convergent Tannakian}
When the base fields $k_i$ are all finite, one can reformulate Theorem~\ref{T:convergent decomposition} in the language of Tannakian fundamental groups to obtain a much more direct analogue of Theorem~\ref{T:Drinfeld}. See \cite[\S 3]{kedlaya-xu}.
\end{remark}

\section{Overconvergent $\Phi$-isocrystals}

We conclude by transferring our previous results, particularly Theorem~\ref{T:convergent decomposition}, from convergent to overconvergent $F$-isocrystals.

\begin{defn}
For $i=1,\dots,n$, let $X_i \to Y_i$ be an open immersion of schemes of finite type over $k_i$
and put
\[
Y := Y_1 \times_{\FF_p} \cdots \times_{\FF_p} Y_n.
\]
We may again define the category
$\PhiaIsoc(X, Y)$ using local liftings as in \cite[\S 2]{kedlaya-isocrystals}; it is again a stack for the Zariski topology and the \'etale topology 
on each $X_i$.

To make this explicit, suppose that $X_i$ and $Y_i$ are both affine.
Let $Q_i$ be an affine formal scheme over $W(k_i)$ with special fiber $Y_i$
which is smooth in a neighborhood of $X_i$, and such that the open subscheme supported on $X_i$ is isomorphic (in a chosen way) to $P_i$.
Let $Q$ be the product $Q_1 \times_{\ZZ_p} \cdots \times_{\ZZ_p} Q_n$.
Let $R$ be the direct limit of $\calO(U)$ as $U$ runs over products of strict neighborhoods of the tube of $X_i$ in the Raynaud generic fiber of $Q_i$ over $K_i$.
Then an object of $\PhiaIsoc(X,Y)$ is a finite projective module $\calE$ over $R$ equipped with an integrable $K$-linear connection
and horizontal isomorphisms $F_i\colon (\sigma_i^a)^* \calE \cong \calE$ that commute pairwise in the same sense as in Definition~\ref{D:phi-isocrystal}.

Again by comparing to a site-theoretic definition using all possible local lifts, one can show that the definition of $\PhiaIsoc(X,Y)$ is independent of the choice of $X$
(compare \cite{lestum}). We may then globalize to allow the case where each $Y_i$ is proper over $k_i$,
and show that the resulting category depends only on the $X_i$ (compare \cite[Lemma~2.6]{kedlaya-isocrystals}.
This gives the category of 
\emph{overconvergent $\Phi^a$-isocrystals} on $X$, denoted $\PhiaIsoc^\dagger(X)$.

The category $\PhiaIsoc(X,Y)$ is a $\QQ_{p^a}$-linear tensor category, where $\QQ_{p^a}$ denotes the unramified
extension of $\QQ_p$ with residue field of degree $a$ over $\FF_p$.
\end{defn}

We have the following analogue of Lemma~\ref{L:fully faithful}.
\begin{lemma} \label{L:fully faithful2}
The restriction functor $\PhiaIsoc(X, Y) \to \PhiaIsoc(X)$ is fully faithful.
\end{lemma}
\begin{proof}
It is sufficient to prove that the restriction functor $\PhiaIsoc(X, Y) \to \PhiaIsoc(X, Y')$ is fully faithful for
$Y' := Y_1 \times_{\FF_p} \cdots \times_{\FF_p} Y_{n-1} \times_{\FF_p} X_n$, as the functor in question can be represented as a composition
of functors of this form.
The proof of Lemma~\ref{L:fully faithful} carries over except that
to show that a given element of $H^0(X, \calE)$ extends, it is not sufficient to pull back
to individual geometric points of $X_1 \times_{\FF_p} \cdots \times_{\FF_p} X_{n-1}$; however, we claim that we may infer
a suitable uniformity over these points which does suffice.

To make this explicit,
we assume that $X = \GG_m^d, Y = \AAA^d$ (the general case can be reduced to this using \cite[Remark~2.9]{kedlaya-isocrystals}).
We then take $P_n := \Spf W(k_n) \langle T_1^{\pm}, \dots, T_d^{\pm} \rangle$, $Q_n := \Spf W(k_n) \langle T_1,\dots,T_n \rangle$.
For $r < 1$, let $U_r$ be the strict neighborhood of the tube of $X_n$ given by the conditions $|T_1|,\dots,|T_n| \geq r$.
Fix a finite set of module generators of $\calE$, let $\calF$ be the free module on these generators, and fix a splitting 
$\calF \cong \calE \oplus \calE'$ of the surjection $\calF \to \calE$.
Let $\widehat{\calE}, \widehat{\calF}$ be the base extensions of $\calE, \calF$ to $\Gamma(P, \calO)$.
Let $A$ be the matrix of action on the basis of $\calF$ for the endomorphism given by projecting to $\calE$, applying $\varphi$ there, and injecting back to $\calE$.
For $r \leq 1$, let $\left| \bullet \right|_r$ be the supremum norm on $\calF$ given by taking the supremum over $U_r$.
Then we can choose $r_0$ such that $|A|_{r_0}, |A^{-1}|_{r_0} > \infty$.

Let $\bv \in \widehat{\calE}$ be the image of an element of $H^0(X, \widehat{\calF})$.
Then $A \varphi(\bv) = \bv$, where now the action of $\varphi$ on $\widehat{\calF}$ is coefficientwise. For $r \in [r_0^p, 1)$, we have
\[
\left| \bv \right|_r \leq \left| \bv \right|_{r^{1/p}} \left| A^{-1} \right|_{r^{1/p}};
\]
on the other hand, by the log-convexity of Gauss norms,
\[
\left| \bv \right|_{r^{1/p}} \leq \left| \bv \right|_r^{1/p} \left| \bv \right|_1^{1-1/p}.
\]
Putting these together yields
\[
\left| \bv \right|_r \leq \left| \bv \right|_1 \left| A^{-1} \right|_{r^{1/p}}^{p/(p-1)}
\]
provided that $\left| \bv \right|_r < \infty$.
The upshot is that given the qualitative statement that $\left| \bv \right|_r < \infty$ for some $r \in (0,1)$, we may then deduce a quantitative bound
for a specific value of $r$; this justifies the argument in the first paragraph.
(Compare \cite[Proposition~2.5.8]{kedlaya-rel}.)
\end{proof}

We have the following analogue of Theorem~\ref{T:relative DM}.

\begin{theorem} \label{T:relative DM overconvergent}
Suppose that $k_n$ is algebraically closed and $X_n = Y_n = \Spec k_n$.
Put $X' := X_1 \times_{\FF_p} \cdots \times_{\FF_p} X_{n-1}$
and $Y' := Y_1 \times_{\FF_p} \cdots \times_{\FF_p} Y_{n-1}$.
Then every object $\calE \in \PhiaIsoc(X,Y)$ decomposes uniquely as a direct sum
\[
\calE \cong \bigoplus_{d \in \QQ} \calE_d
\]
in which for $d = \frac{r}{s}$ in lowest terms,
$\calE_d$ is obtained by pulling back an object of $\PhiaIsoc(X', Y)$ equipped with an
endomorphism $F_n$ such that $F_n^s = p^r$, which then gives the action of $\varphi_n$ on the pullback.
\end{theorem}
\begin{proof}
By Theorem~\ref{T:relative DM}, we obtain a corresponding decomposition in $\PhiaIsoc(X)$;
we may then use Lemma~\ref{L:fully faithful2} to promote this decomposition to $\Phi^a\Isoc(X, Y)$.
\end{proof}

We have the following analogue of Lemma~\ref{L:pushforward}, but with a very different proof.
\begin{lemma} \label{L:pushforward2}
Let $\pi_n\colon X \to X_n$ be the projection map.
\begin{enumerate}
\item[(a)]
The pullback functor $\pi_n^*$ from $\FaIsoc^\dagger(X_n)$ to $\PhiaIsoc^\dagger(X)$ admits a left and right adjoint $\pi_{n*}$.
\item[(b)]
There exists an open dense subset $V_n$ of $X_n$ such that the formation of $\pi_{n*}(\calE)$ commutes with base change from $V_n$ to any other scheme.
\end{enumerate}
\end{lemma}
\begin{proof}
Apply \cite[Theorem~7.3.3]{kedlaya-fin}.
\end{proof}

\begin{lemma} \label{L:product overconvergent}
Suppose that $\calE$ is absolutely irreducible. Then 
there exist absolutely irreducible objects $\calF \in \PhiaIsoc(X'), \calG \in \FaIsoc(X_n)$  such that $\calE \cong \calF \boxtimes \calG$.
\end{lemma}
\begin{proof}
As per Remark~\ref{R:reduce to first power}, we are free to increase $a$. Using Theorem~\ref{T:relative DM overconvergent} in place of Theorem~\ref{T:relative DM},
for $r \in \ZZ$
we may define the \emph{$r$-slice} of $\calE$ over a geometric point $\overline{x}_n$ of $X_n$.
By Lemma~\ref{L:fully faithful} plus Lemma~\ref{L:same slices}, the $r$-slices of $\calE$ over any two geometric points of $U_n$ are isomorphic.
By increasing $a$, we may ensure that there exists $r \in \ZZ$ for which one of these slices, which we denote by $\calE'_r$, is nonzero.
Choose $V_n \subseteq U_n$ as in Lemma~\ref{L:pushforward2};
we can then define $\calG = \pi_{n*}(\pi_n^{\prime *}(\calE_r^{\prime \dual}) \otimes \calE) \in \FaIsoc^\dagger(V_n)$.
Arguing as in the proof of Lemma~\ref{L:convergent decomposition}, 
we obtain a nonzero map $\calE'_r \boxtimes \calG \to \calE|_{X' \times_{\FF_p} V_n}$.

We now argue as in the proof of Lemma~\ref{L:convergent decomposition2}.
Pulling back from $X'$ to a product of geometric points $\overline{x}'$,
we obtain a surjective morphism $\calF|_{\overline{x}'} \boxtimes \calG \to \calE|_{\overline{x}' \times_{\FF_p} V_n}$.
By decomposing $\calE|_{\overline{x}' \times_{\FF_p} X_n}$ according to Theorem~\ref{T:relative DM overconvergent},
we obtain a surjective morphism $\calG^{\oplus r} \to \calH|_{V_n}$ in $\FaIsoc^\dagger(V_n)$
for some nonzero $\calH \in \FaIsoc(X_n)$. In particular, all of the constituents of $\calH$ in $\FaIsoc^\dagger(V_n)$
are isomorphic to $\calG$, so there exists a surjection $\calH|_{U_n} \to \calG$. 
We may thus replace $\calG$ with an absolutely irreducible quotient $\calG'$ of $\calH$ in $\FaIsoc^\dagger(X_n)$,
then apply Lemma~\ref{L:fully faithful2} to extend the resulting morphism $\calF \boxtimes (\calG'|_{V_n}) \to \calE|_{X' \times_{\FF_p} V_n}$ to a morphism $\calF \boxtimes \calG' \to \calE$. As in the proof of Lemma~\ref{L:convergent decomposition2}, this map is an isomorphism.
\end{proof}

\begin{theorem} \label{T:product overconvergent}
Any absolutely irreducible object in $\PhiaIsoc^\dagger(X)$ is isomorphic to
$\calE_1 \boxtimes \dots \boxtimes \calE_n$ for some absolutely irreducible objects $\calE_i \in \FaIsoc^\dagger(X_i)$.
\end{theorem}
\begin{proof}
This follows from Lemma~\ref{L:product overconvergent} by induction on $n$.
\end{proof}

\begin{remark}
Echoing Remark~\ref{R:convergent Tannakian}, in the case where the $k_i$ are all finite (which is the only case needed for the application to geometric Langlands) one can reformulate
Theorem~\ref{T:product overconvergent} as a statement about products of Tannakian fundamental groups. In addition, one can use cohomological methods to establish this result directly,
without reduction to the convergent case. See \cite[\S 1]{kedlaya-xu}.
\end{remark}

\end{document}